\RequirePackage{fix-cm}
\documentclass{article}       % onecolumn (second format)
\usepackage{enumerate}
\usepackage{amsmath}
\usepackage{mathtools}
\usepackage{graphicx}
\usepackage{amsfonts}

\usepackage{amsthm}
\theoremstyle{definition}
\newtheorem{defn}{Definition}[section]
\numberwithin{defn}{section}

\theoremstyle{plain}

\newtheorem{thm}[defn]{Theorem}
\newtheorem{lem}[defn]{Lemma}
\newtheorem{prop}[defn]{Proposition}

\begin{document}

\title{Asymptotic Behaviour of the Conjugacy Probability of the Alternating Group%\thanks{Grants or other notes
}


\author{Misja F.A. Steinmetz         \and
        Madeleine L. Whybrow %etc.
}


\date{22 December 2015}

\maketitle

\begin{abstract}
For $G$ a finite group, $\kappa(G)$ is the probability that that $\sigma, \tau \in G$ are conjugate, when $\sigma$ and $\tau$ are chosen independently and uniformly at random. Recently, Blackburn \emph{et al} (2012) gave an elementary proof that $\kappa(S_n) \sim A/n^2$ as $n \to\infty$ for some constant $A$ - a result which was first proved by Flajolet \emph{et al} (2006).  In this paper, we extend the elementary methods of Blackburn \emph{et al} to show that $\kappa(A_n) \sim B/n^2$ as $n \to\infty$ for some constant $B$, given explicitly in this paper.
\end{abstract}

\section{Introduction}
\label{intro}
Let $G$ be a finite group and $\kappa(G)$ be the probability that $\sigma, \tau \in G$ are conjugate, when $\sigma$ and $\tau$ are chosen independently and uniformly at random. In words, $\kappa(G)$ is the probability that two randomly chosen elements from the group are conjugate. 

If $G$ is any group and $g_1,g_2,\dots,g_k$ is a complete set of representatives for the conjugacy classes of $G$, then it is easy to see that 

\[
\kappa(G)=\frac{1}{|G|^2}\sum_{i=1}^{k}|g_i^G|^2=\sum_{i=1}^{k}\frac{1}{|\mathrm{Cent}_G(g_i)|^2},
\]

where $\mathrm{Cent}_G(g)$ denotes the centralizer of an element $g\in G$. 

It was first proved by Flajolet \emph{et al} (2006) and later, using more elementary methods, by Blackburn \emph{et al} (2012) that $\kappa (S_n)\sim A/n^2$ as $n\to\infty$ for some constant $A$. In this paper, we extend the methods used by Blackburn \emph{et al} (2012, pp.~11-17) to show that $\kappa(A_n)\sim B/n^2$ as $n\to\infty$ for some constant $B$.

Although parts of this paper use similar methods to those of Blackburn \emph{et al}, some complications arise when treating the alternating groups instead of the symmetric groups. In particular, a conjugacy class of $S_n$ splits into two classes in $A_n$ if its corresponding cycle type consists of cycles of distinct odd lengths. 

In addressing this problem, we introduce the following definitions:

\begin{defn}\label{defn1.1}
Suppose $\sigma,\tau\in S_n$ are chosen independently and uniformly at random. Let
\begin{itemize}
	\item $\kappa_E(S_n)$ be the probability that $\sigma,\tau$ are conjugate, given that they are even permutations,
	\item$\kappa_O(S_n)$ be the probability that $\sigma,\tau$ are conjugate, given that they are odd permutations,
	\item $Q(S_n)$ be the probability that $\sigma,\tau$ have the same cycle type and they are only composed of cycles of distinct odd lengths.
\end{itemize}
\end{defn}

In the remainder of this paper we shall adopt the standard convention that $S_0$ equals the trivial group. One should note, however, that $\kappa_O(S_n)$ is not well defined for $n=0,1$, since with our convention neither $S_0$ nor $S_1$ contains any odd permutations. In order for Proposition \ref{prop1} in Section \ref{section1} to make sense, we define $\kappa_O(S_0):=1$ and $\kappa_O(S_1):=0.$ Please note that all of the three statistics given above are now well defined on $S_n$ for all $n\geq 0.$

We find that the asymptotic behavior of $\kappa(A_n)$ depends on the parity of $n$. We thus split our result into two cases, the first when the limit of $n^2 \kappa(A_n)$ is taken over the even integers and the second when it is taken over the odd integers. 

We now state our main result:

\begin{thm}
\label{thm1}

Using notation as defined above,

\[
\lim_{\substack{n \to\infty\\ n \textrm{ even}}} n^2 \kappa(A_n) = \sum_{\substack{d=0\\d \textrm{ even}}}^{\infty} \kappa_O(S_d) + \sum_{\substack{d=1\\ d \textrm{ odd}}}^{\infty} \left(\kappa_E(S_d)-2Q(S_d) \right)
\]
and
\[
\lim_{\substack{n \to\infty\\ n \textrm{ odd}}} n^2 \kappa(A_n) = \sum_{\substack{d=0\\d \textrm{ even}}}^{\infty} \left(\kappa_E(S_d) - 2Q(S_d)\right) + \sum_{\substack{d=1\\ d \textrm{ odd}}}^{\infty} \kappa_O(S_d).
\]
\end{thm}
The remainder of this paper concerns the proof of this result. In Section \ref{section1} we prove that for $\kappa_E(S_n), \kappa_O(S_n)$ and $Q(S_n)$ we can find inequalities iteratively relating their value to sums over values of these statistics at $S_k$ for smaller $k.$

We then use these inequalities in Section \ref{bound} to find constants $C_0,C_1$ and $C_2$ such that for all $n\in\mathbb{N}$ $$\kappa_E(S_n)\leq \frac{C_0}{n^2},\kappa_O(S_n)\leq\frac{C_1}{n^2}\text{ and } Q(S_n)\leq\frac{C_2}{n^2}.$$ In other words, we will establish uniform bounds on our probabilities in this section.

In Section \ref{section4} we find expressions for the limits $$\lim_{n\to\infty}n^2\kappa_E(S_n)\text{ and }\lim_{n\to\infty}n^2 Q(S_n).$$

Finally in Section \ref{section5} we show that we can use our results for $\kappa_E(S_n), \kappa_O(S_n)$ and $Q(S_n)$ to understand the asymptotic behaviour of $\kappa(A_n)$ and prove Theorem \ref{thm1}.

\section{Inequalities on $\kappa_{E}(S_n)$ and $Q(S_n)$}\label{section1}

In this section we will obtain analogous results to Proposition 7.1 of Blackburn \emph{et al} (2012, p.~13) for $\kappa_E(S_n)$ and $Q(S_n)$. We first state the following well known lemma found in Blackburn \emph{et al} (2012).

\begin{lem}
\label{prelem1}
Let $n \in \mathbb{N}$ and let $1 \leq l \leq n$. Let $X \in \Omega_n$ be an $l$-set. If $\sigma$ is chosen uniformly and at random from $S_n$ then
\begin{itemize}
\item the probability that $\sigma$ acts as an $l$-cycle on $X$ is $\frac{1}{l}\binom{n}{l}^{-1}$;
\item the expected number of $l$-cycles in $\sigma$ is $1/l$;
\item the probability that 1 is contained in an $l$-cycle is 1/n.
\end{itemize}
\end{lem}

First we will turn our attention to $\kappa_E(S_n)$. We write $s_k(n)$ for the probability that a permutation of $\Omega_n$, chosen at random, has only cycles of length strictly less that $k$.

\begin{prop}\label{prop1}
For all $n \in \mathbb{N}$ we have 
\[
\kappa_E(S_n) \leq s_k(n)^2+\sum_{\substack{l=k \\ l \textrm{ even}}}^{n} \frac{\kappa_O(S_{n-l})}{l^2} + \sum_{\substack{l=k \\ l \textrm{ odd}}}^{n} \frac{\kappa_E(S_{n-l})}{l^2}.
\]
Moreover, if $k$ is such that $n/2 < k \leq n$ then 
\[
\kappa_E(S_n) \geq \sum_{\substack{l=k \\ l \textrm{ even}}}^{n} \frac{\kappa_O(S_{n-l})}{l^2} + \sum_{\substack{l=k \\ l \textrm{ odd}}}^{n} \frac{\kappa_E(S_{n-l})}{l^2}.
\]
\end{prop}

\begin{proof}

Let $\sigma$ and $\tau$ be even permutations of $\Omega_n$ chosen independently and uniformly at random. Let $X$ and $Y$ be two $l$-sets and write $\overline{X}$ and $\overline{Y}$ for their respective complements in $\Omega_n$. 

Let $E(X,Y)$ be the event that $\sigma$ acts as an $l$-cycle on $X$, $\tau$ acts as an $l$-cycle on $Y$ and that $\overline{\sigma}$ and $\overline{\tau}$, the respective restrictions of $\sigma$ and $\tau$ to $\overline{X}$ and $\overline{Y}$, have the same cycle structure. 

If $l$ is even then $\overline{\sigma}$ and $\overline{\tau}$ must be odd so the probability that they are of the same cycle type is equal to $\kappa_O(S_{n-l})$. If $l$ is odd, the probability is $\kappa_E(S_{n-l})$. So, from Lemma \ref{prelem1}, we have 
\[
\mathbf{P}(E(X,Y)) = 
	\begin{cases}
		\binom{n}{l}^{-2}\frac{\kappa_O(S_n)}{l^2}, & l \text{ even}; \\
		\binom{n}{l}^{-2}\frac{\kappa_E(S_n)}{l^2}, & l \text{ odd}.
	\end{cases}	
\]
If $\sigma$ and $\tau$ are conjugate in $S_n$ then either $\sigma$ and $\tau$ both contain only cycles of length strictly less than $k$, or there exists sets $X$ and $Y$ of cardinality $l \geq k$ on which $\sigma$ and $\tau$ act as $l$-cycles and such that the restrictions $\overline{\sigma}$ and $\overline{\tau}$ have the same cycle type. Therefore
\begin{align*}
\kappa_E(S_n) &\leq s_k(n)^2 +\sum_{\substack{l=k \\ l \textrm{ even}}}^{n} \sum_{|X|=l} \sum_{|Y|=l} \mathbf{P}(E(X,Y))+\sum_{\substack{l=k \\ l \textrm{ odd}}}^{n} \sum_{|X|=l} \sum_{|Y|=l} \mathbf{P}(E(X,Y))\\
&=s_k(n)^2+\sum_{\substack{l=k \\ l \textrm{ even}}}^{n} \sum_{|X|=l} \sum_{|Y|=l} \binom{n}{l}^{-2}\frac{\kappa_O(S_{n-l})}{l^2}\\ &\hspace{3cm}+ \sum_{\substack{l=k \\ l \textrm{ odd}}}^{n} \sum_{|X|=l} \sum_{|Y|=l} \binom{n}{l}^{-2}\frac{\kappa_E(S_{n-l})}{l^2}\\
&=s_k(n)^2+\sum_{\substack{l=k \\ l \textrm{ even}}}^{n} \frac{\kappa_O(S_{n-l})}{l^2}+ \sum_{\substack{l=k \\ l \textrm{ odd}}}^{n} \frac{\kappa_E(S_{n-l})}{l^2}.
\end{align*}
Here the events $E(X,Y)$ are not necessarily disjoint so we have only an upper bound, not an equality. We have established the first inequality of the proposition.

When $k >\frac{n}{2}$ the events $E(X,Y)$ with $|X|,|Y| \geq k$ are disjoint, since a permutation of length $n$ can only contain at most one cycle of length greater than $\frac{n}{2}$. Thus
\begin{align*}
\kappa_E(S_{n}) 
&=s_k(n)^2+\sum_{\substack{l=k \\ l \textrm{ even}}}^{n} \frac{\kappa_O(S_{n-l})}{l^2} + \sum_{\substack{l=k \\ l \textrm{ odd}}}^{n} \frac{\kappa_E(S_{n-l})}{l^2}\\
&\geq\sum_{\substack{l=k \\ l \textrm{ even}}}^{n} \frac{\kappa_O(S_{n-l})}{l^2} + \sum_{\substack{l=k \\ l \textrm{ odd}}}^{n} \frac{\kappa_E(S_{n-l})}{l^2},
\end{align*}
as required. 
\end{proof}

We now take a closer look at $Q(S_n)$. Recall that $Q(S_n)$ is the probability that two elements of $S_n$, chosen independently and uniformly at random, have the same cycle type and consist of cycles of distinct odd lengths. 

\begin{prop}\label{prop2}

For all  $n \in \mathbb{N}$ we have $$Q(S_n)\leq s_k(n)^2 + \sum_{\substack{l=k \\ l \mathrm{\;odd}}}^{n} \frac{Q(S_{n-l})}{l^2}.$$
Moreover, if $k$ is such that $\frac{n}{2}<k\leq n$, then $$Q(S_n) \geq \sum_{\substack{l=k\\l \mathrm{\;odd}}}^{n} \frac{Q(S_{n-l})}{l^2}.$$

\end{prop}

\begin{proof}
Let $l$ be odd and let $\sigma,\tau \in S_n$. We define $X,Y,\overline{X},\overline{Y},\overline{\sigma}$ and $\overline{\tau}$ as in the proof of Proposition \ref{prop1}.

Let $F(X,Y)$ be defined as the event that $\sigma$ acts as an $l-$cycle on $X$, $\tau$ acts as an $l-$cycle on $Y$, $\overline{\sigma}$ and $\overline{\tau}$ have the same cycle type and $\overline{\sigma}$ and $\overline{\tau}$ only have cycles of distinct odd length. Given that $\sigma$ and $\tau$ act as $l-$cycles on $X$ and $Y$, $\overline{\sigma}$ and $\overline{\tau}$ are independently and uniformly distributed over the symmetric groups of $\overline{X}$ and $\overline{Y}$. Hence the probability that $\overline{\sigma}$ and $\overline{\tau}$ have the same cycle type and $\overline{\sigma}$ and $\overline{\tau}$ only have cycles of distinct odd length is exactly $Q(S_{n-l})$.  So, from Lemma \ref{prelem1}, $$\mathbf{P}(F(X,Y))=\binom{n}{l}^{-2}\frac{Q(S_{n-l})}{l^2}.$$

If $\sigma$ and $\tau$ have the same cycle type and only parts of distinct odd lengths, then either both have only cycles of length strictly less than $k$, or there exist $l-$sets $X$ and $Y$ for some $l\geq k$ such that $\sigma$ acts on $X$ as an $l-$cycle and $\tau$ acts on $Y$ as an $l-$cycle and such that the restrictions $\overline{\sigma}$ and $\overline{\tau}$ have the same cycle type and only parts of distinct odd lengths. Therefore
\begin{align*}
Q(S_n)
& \leq s_k(n)^2 + \sum_{\substack{l=k\\ l \textrm{ odd}}}^{n} \sum_{|X|=l}\sum_{|Y|=l} \mathbf{P}(F(X,Y)) \\
& = s_k(n)^2 + \sum_{\substack{l=k\\ l \textrm{ odd}}}^{n} \sum_{|X|=l}\sum_{|Y|=l}  \binom{n}{l}^{-2} \frac{Q(S_{n-l})}{l^2} \\
& = s_k(n)^2 + \sum_{\substack{l=k\\ l \textrm{ odd}}}^{n} \frac{Q(S_{n-l})}{l^2}.
\end{align*}
Here, again, the $F(X,Y)$ are not necessarily independent, meaning we do not have equality. The inequality also allows us to ignore the case where $\overline{\sigma}$ contains a further $l$-cycle. We have established the first inequality of the proposition. 

When $k>\frac{n}{2}$ the events $F(X,Y)$ with $|X|,|Y| \geq k$ are disjoint, since a permutation of length $n$ can only contain at most one cycle of length greater than $\frac{n}{2}$. What's more, $\overline{\sigma}$ may not contain an $l$-cycle. Thus
\begin{align*}
Q(S_n)
& = s_k(n)^2 + \sum_{\substack{l=k\\ l \textrm{ odd}}}^{n} \frac{Q(S_{n-l})}{l^2} \\
& \geq \sum_{\substack{l=k\\ l \textrm{ odd}}}^{n} \frac{Q(S_{n-l})}{l^2}.
\end{align*} 

\end{proof}

\section{Uniform Bounds on $\kappa_E(S_n)$, $\kappa_O(S_n)$ and $Q(S_n)$}\label{bound}

In this section we establish an analogous result to Theorem 1.4 in Blackburn \emph{et al} (2012, p.~3). We will prove that there exist $C_{0}$, $C_{1}$ and $C_{2}$ such that for all $n\in\mathbb{N}$ $$ \kappa_E(S_n) \leq \frac{C_{0}}{n^2},\: \kappa_O(S_n) \leq \frac{C_{1}}{n^2},\: Q(S_n) \leq \frac{C_{2}}{n^2}.$$
\\
We will need the following results, found in Blackburn \emph{et al} (2012, pp.~13--14).

\begin{prop}
\label{preprop1}
Let $k \in \mathbb{N}$ be such that $k \geq 2$. Suppose that there exists $n_0 \in \mathbb{N}$ such that 
\[
ns_k(n) \geq (n+1)s_k(n)(n+1) \textrm{ for } n \in \{n_0,n_0+1,\ldots,n_0+k-2\}.
\]
Then $ns_k(n) \geq (n+1)s_k(n+1)$ for all $n \geq n_0$.
\end{prop}

\begin{lem}
\label{prelem2}
Let $n \in \mathbb{N}$ and let $0 <n<k/2$. Then 

\[
\sum_{l=\lceil n/2 \rceil}^{n-k-1} \frac{1}{l^2(n-l)^2}\leq \frac{1}{n^2k} + \frac{2 \log(n/k)}{n^3}.
\]
\end{lem}

For the sake of brevity we have decided to omit proofs of these results. They can be found in Blackburn \emph{et al} (2012, pp.~13--15).

We are now ready to start showing that bounds as given above exist. We first consider $\kappa_E(S_n)$ and $\kappa_O(S_n)$. For all $n\geq 2$,
\begin{align*}
& \kappa_E(S_n) + \kappa_O(S_n)=4\kappa(S_n),\\
& 0\leq\kappa_E(S_n),\kappa_O(S_n)\leq 1 \textrm{ and}\\
& n^2\kappa(S_n)\leq C_{\kappa},
\end{align*}
where $C_{\kappa}:=13^2\kappa(S_{13})$ (Blackburn \emph{et al}, 2012, p.~15). Hence, we find that, for all $n\in\mathbb{N},$
\begin{align*}
&n^2\kappa_E(S_n)\leq 4C_{\kappa}\text{ and}\\
&n^2\kappa_O(S_n)\leq 4C_{\kappa},
\end{align*}
which establishes the required upper bounds for $\kappa_E(S_n)$ and $\kappa_O(S_n).$
\\
\\
We now consider the case of $Q(S_n)$. In this case, we can do better; we may follow the same argument as Blackburn \emph{et al} (2012, pp.~14-16) to find a tight upper bound for $n^{2}Q(S_n)$. In fact, we will prove that this upper bound is achieved at $n=4.$ We first need a result similar to Lemma 8.2 in Blackburn \emph{et al} (2012, p.~15). We define $C_2:=4^2Q(S_4).$

\begin{lem}\label{lem1} We have
\begin{description}
\item[(i)]
$Q(S_n)\leq C_2/n^2 \textrm{ for all } n\leq 300;$

\item[(ii)]
$\displaystyle \sum_{\substack{m=0\\m\textrm{ even}}}^{15} Q(S_m)=\frac{630468719}{521756235}<1.20836;$

\item[(iii)]
$\displaystyle \sum_{\substack{m=1\\m\textrm{ odd}}}^{15} Q(S_m)=\frac{4429844723}{3652293645}<1.21290;$

\item[(iv)]
$C_2=\frac{16}{9}<1.77778.$

\end{description}
\end{lem}

\begin{proof}
Most of these results are straightforward calculations. To prove part (i) we introduce the generating function $$\displaystyle 1+\sum_{n=1}^{\infty} Q(S_n)x^n =\prod_{d \textrm{ odd}} \left(1+\frac{x^d}{d^2}\right).$$ We easily see that this is indeed the required generating function as the coefficient of $x^n$ after multplying out the brackets will be $\sum_{S} \left (\prod_{s\in S} \frac{1}{s^2}\right )$ where the sum over all sets $S$ such that its elements are distinct odd integers and $\sum_{s\in S} s = n$. Parts (ii), (iii) and (iv) are straightforward calculations. 
\end{proof}

Two further results can be found in Lemma 8.2 of Blackburn \emph{et al} (2012, p.~15):

\begin{description}
\item[(i)] $60s_{15}(60) = \frac{158929798034197186400893117108816122671}{83317523526667097802976844202788608000} < 0.19076$
\item[(ii)] $ns_{15}(n) \geq (n+1)s_{15}(n+1) \textrm{ for } 14 \leq n \leq 60$.
\end{description}

We can now prove our main proposition in a similar way to the proof of Theorem 1.4 in Blackburn \emph{et al} (2012, pp.~15-16).

\begin{prop}
$Q(S_n)\leq C_2/n^2 \textrm{ for all } n\in\mathbb{N}.$ Moreover, since we defined $C_2:=4^2 Q(S_4)$, this bound is tight.
\end{prop}

\begin{proof}
We proceed by induction on $n$. By Lemma \ref{lem1}(i) the proposition holds if $n\leq 300$, and so we may conclude that $n>300$. By Proposition \ref{prop1} in the case $k=15$ we have
$$Q(S_n)\leq s_{15}(n)^2 +\sum_{\substack{l=15\\l\textrm{ odd}}}^{n} \frac{Q(S_{n-l})}{l^2}$$ 
and hence 
\begin{equation}
\label{bdeq1}
n^2 Q(S_n)\leq n^2 s_{15}(n)^2+n^2 \sum_{\substack{l=n-15\\l\textrm{ odd}}}^{n} \frac{Q(S_{n-l})}{l^2} +n^2 \sum_{\substack{l=15\\l\textrm{ odd}}}^{n-16} \frac{Q(S_{n-l})}{l^2}.
\end{equation}
It follows from Proposition \ref{preprop1} and the comments preceding this proposition that $ns_{15}(n) \leq 60s_{15}(60) <0.19076$. Hence $n^2s_{15}(n)^2 \leq 0.03639$.

%Approximating the three terms in this equation in the same way as in \cite[p./~19]{john} we find that $s_{15}(n)^2\leq0.03639$. Using \cite[p./~19]{john} again, we find that 

Using Lemma \ref{lem1}(ii) to bound the second term in (\ref{bdeq1}), we get
\begin{align*}n^2 \sum_{\substack{l=n-15\\l\textrm{ odd}}}^{n} \frac{Q(S_{n-l})}{l^2}&\leq \left ( \frac{n}{n-15} \right)^2\sum_{\substack{m=1\\m\textrm{ odd}}}^{15} Q(S_m)\\ &\leq \left(\frac{300}{285}\right)^{2}\sum_{\substack{m=1\\m\textrm{ odd}}}^{15} Q(S_m)\\ &\leq 1.34393.\end{align*} 
It is clear from Lemma \ref{lem1}(ii) and (iii) that the sum over odd values of $Q(S_n)$ is the larger of the two values, and thus the one we take for our bound. 

%Here we can forget about the case in which we have to sum over even $m$ because by Lemma \ref{lem1} that would give us a smaller value anyway.

For the third term in (\ref{bdeq1}) we use the inductive hypothesis to get

\[
n^2 \sum_{\substack{l=15\\l \textrm{ odd}}}^{n-16} \frac{Q(S_{n-l})}{l^2} \leq n^2 \sum_{\substack{l=15\\l \textrm{ odd}}}^{n-16} \frac{C_2}{l^2(n-l)^2}.
\]
%Lastly approximating the third term in the exact same way as in the paper \cite[p./~19]{john} we find that 
Using the symmetry of $l^2(n-l)^2$ in this sum and then applying Lemma \ref{prelem2} in the case $n=15,$ we get
\begin{align*}
n^2 \sum_{\substack{l=15\\l \textrm{ odd}}}^{n-16}\frac{1}{l^2(n-l)^2} &\leq n^2 \sum_{l=15}^{n-16} \frac{1}{l^2(n-l)^2} \\
&\leq 2n^2 \sum_{\lceil n/2 \rceil}^{n-16} \frac{1}{l^2(n-l)^2} + \frac{n^2}{15^2(n-15)^2}\\
&\leq 2 \left(\frac{1}{15} + \frac{2\log(n/15)}{n} \right) + \frac{1}{15^2}\left(\frac{300}{285}\right)^2.
\end{align*}
Since $\log(n/15)/n$ is decreasing for $n>40$, it follows from the upper bound for $C_2$ in Lemma \ref{lem1}(iv) that
$$n^2 \sum_{\substack{l=15\\l\textrm{ odd}}}^{n-16} \frac{Q(S_{n-l})}{l^2}\leq C_2\left (\frac{2}{15}+\frac{4\log(300/15)}{300}+\frac{300^2}{15^2\cdot 285^2}\right)\leq 0.31681.$$
Hence $$ n^2 Q(S_n)\leq 0.03639+1.34393+ 0.31681=1.69713<C_2$$ and the proposition follows. 

\end{proof}

\section{Limits}\label{section4}
In this section we will prove asymptotic results on $\kappa_E(S_n)$, $Q(S_n)$ and $\kappa(A_n)$ analogous to Theorem 1.5 in Blackburn \emph{et al} (2012, p.~3).

We will need the following lemma, found in Blackburn \emph{et al} (2012, p.~12). Recall that we write $s_k(n)$ for the probability that a permutation of $\Omega_n$, chosen at random, has only cycles of length strictly less that k.

\begin{lem}
\label{preprop2}
For all $n,k \in \mathbb{N}$ with $k \geq 2$ we have $s_k(n) \leq 1/t! < \left(\frac{e}{t}\right)^t$ when $t = \lfloor n/(k-1) \rfloor$.
\end{lem} 

%Firstly, let's set 

%\begin{align*}
%A_{1} & = \sum_{\substack{d=0\\d \textrm{ even}}}^{\infty} \kappa_O(S_d) + \sum_{\substack{d=1\\ d \textrm{ odd}}}^{\infty} \kappa_E(S_d)\\
%A_{2} & = \sum_{\substack{d=0\\d \textrm{ even}}}^{\infty} \kappa_E(S_d) + \sum_{\substack{d=1\\ d \textrm{ odd}}}^{\infty} \kappa_O(S_d)\\
%\end{align*}

For brevity, we introduce the following notation:
\begin{align*}
A_{1} & = \sum_{\substack{d=0\\d \textrm{ even}}}^{\infty} \kappa_O(S_d) + \sum_{\substack{d=1\\ d \textrm{ odd}}}^{\infty} \kappa_E(S_d);\\
A_{2} & = \sum_{\substack{d=0\\d \textrm{ even}}}^{\infty} \kappa_E(S_d) + \sum_{\substack{d=1\\ d \textrm{ odd}}}^{\infty} \kappa_O(S_d);\\
B_{1} & = \sum_{\substack{d=1\\ d \textrm{ odd}}}^{\infty} Q(S_d);\\
B_{2} & = \sum_{\substack{d=0\\ d \textrm{ even}}}^{\infty} Q(S_d).
\end{align*}

\begin{prop}\label{prop3}
\[
\liminf_{\substack{n \rightarrow \infty \\ n \textrm{ even}}} n^2 \kappa_E(S_n) \geq \displaystyle A_{1},
\]

\[
\liminf_{\substack{n \rightarrow \infty \\ n \textrm{ odd}}} n^2 \kappa_E(S_n) \geq \displaystyle A_{2}.
\]

\end{prop}

\begin{proof}
The second part of Proposition \ref{prop1} says that, if $k > n/2$ then 

\[
\kappa_E(S_n) \geq \sum_{\substack{l=k \\ l \textrm{ even}}}^{n} \frac{\kappa_O(S_{n-l})}{l^2} + \sum_{\substack{l=k \\ l \textrm{ odd}}}^{n} \frac{\kappa_E(S_{n-l})}{l^2}.
\]
Hence

\[
n^2 \kappa_E(S_n)\geq  \sum_{\substack{l=k \\ l \textrm{ even}}}^{n} \kappa_O(S_{n-l}) + \sum_{\substack{l=k \\ l \textrm{ odd}}}^{n} \kappa_E(S_{n-l}).
\]

Taking $k=\left\lfloor 3n/4 \right\rfloor$ and letting $n \rightarrow \infty$ we see that

\[
\liminf_{\substack{n \rightarrow \infty \\ n \textrm{ even}}} n^2 \kappa_E(S_n) \geq \displaystyle \sum_{\substack{d=0\\d \textrm{ even}}}^{\infty} \kappa_O(S_d) + \sum_{\substack{d=1\\ d \textrm{ odd}}}^{\infty} \kappa_E(S_d) = A_{1},
\]

\[
\liminf_{\substack{n \rightarrow \infty \\ n \textrm{ odd}}} n^2 \kappa_E(S_n) \geq \displaystyle \sum_{\substack{d=0\\d \textrm{ even}}}^{\infty} \kappa_E(S_d) + \sum_{\substack{d=1\\ d \textrm{ odd}}}^{\infty} \kappa_O(S_d)=A_2.
\] 

%where the $+1$ in $A_2$ is caused by $n-$cycles being important for the limit, while $S_0$ is undefined.

\end{proof}

\begin{prop}\label{prop4}
\[
\limsup_{\substack{n \rightarrow \infty \\ n \textrm{ even}}} n^2 \kappa_E(S_n) \leq
 \displaystyle A_{1}, 
\]

\[
\limsup_{\substack{n \rightarrow \infty \\ n \textrm{ odd}}} n^2 \kappa_E(S_n) \leq \displaystyle A_{2}.
\]
\end{prop}

\begin{proof}
Let $k=\left\lfloor \frac{n}{\log(n)} \right\rfloor$. By Proposition \ref{prop1} we have 
\[
\kappa_E(S_n) \leq s_k(n)^2+\sum_{\substack{l=k \\ l \textrm{ even}}}^{n} \frac{\kappa_O(S_{n-l})}{l^2} + \sum_{\substack{l=k \\ l \textrm{ odd}}}^{n} \frac{\kappa_E(S_{n-l})}{l^2}.
\] 

By Lemma \ref{preprop2} we have
\[
s_k(n) < \left( \frac{e}{t} \right)^t
\]
where $t= \lfloor \frac{n}{k-1} \rfloor$. Writing $k= n/\log n + O(1) $ we have $$\lfloor \frac{n}{k-1} \rfloor = (\log n)\left(1+O(\frac{\log n}{n})\right),$$ and so 
\begin{align*}
\log(ns_k(n)) &< \log (n) t(1-\log t )\\
&= \log n - \log n \log \log n + \log \left(1+O\left(\frac{\log n}{n}\right)\right)\\
& \rightarrow -\infty
\end{align*}
as $n \rightarrow \infty$. Hence $ns_k(n) \rightarrow 0$ as $n \rightarrow \infty$.

We estimate the main sum in the same way as in Blackburn \emph{et al} (2012, p.~17). This gives
\begin{align}
&\sum_{\substack{l=k \\ l \textrm{ even}}}^{n} \frac{\kappa_O(S_{n-l})}{l^2} + \sum_{\substack{l=k \\ l \textrm{ odd}}}^{n} \frac{\kappa_E(S_{n-l})}{l^2} \notag\\ 
&\leq \sum_{\substack{l=n-k \\ l \textrm{ even}}}^{n} \frac{\kappa_O(S_{n-l})}{l^2} + \sum_{\substack{l=n-k \\ l \textrm{ odd}}}^{n} \frac{\kappa_E(S_{n-l})}{l^2} +  \sum_{\substack{l=k \\ l \textrm{ even}}}^{n-k-1} \frac{\kappa_O(S_{n-l})}{l^2} + \sum_{\substack{l=k \\ l \textrm{ odd}}}^{n-k-1} \frac{\kappa_E(S_{n-l})}{l^2} \notag \\
&\leq  \sum_{\substack{l=n-k \\ l \textrm{ even}}}^{n} \frac{\kappa_O(S_{n-l})}{l^2} + \sum_{\substack{l=n-k \\ l \textrm{ odd}}}^{n} \frac{\kappa_E(S_{n-l})}{l^2} +  \sum_{l=k}^{n-k-1} \frac{C'}{l^2(n-l)^2} \label{lim2}\\
&\leq \sum_{\substack{l=n-k \\ l \textrm{ even}}}^{n} \frac{\kappa_O(S_{n-l})}{l^2} + \sum_{\substack{l=n-k \\ l \textrm{ odd}}}^{n} \frac{\kappa_E(S_{n-l})}{l^2} + \sum_{l=\left\lceil n/2\right\rceil}^{n-k-1} \frac{2C'}{l^2(n-l)^2} + \frac{C'}{k^2(n-k)^2}, \label{lim3}
\end{align}
where equation (\ref{lim2}) is justified by the bounds in Section \ref{bound}, $C'$ is defined as $\max\{C_{0},C_1\}$ and equation (\ref{lim3}) is as a result of the symmetry in $l^2(n-l)^2$.

By Lemma \ref{prelem2} the third term in equation (\ref{lim3}) is at most $(2C'\log n)/n^3$. Moreover, from the identity $\frac{n}{k(n-k)}=\frac{1}{k} + \frac{1}{n-k}$ it is clear that $\frac{n^2}{k^2(n-k)^2} \rightarrow 0 \textrm{ as } n \rightarrow \infty$. So the last two terms in equation (\ref{lim3}) are $o(n^{-2})$ and may safely be ignored.\\
\\
We now consider the case when $n$ is even. Let $\epsilon \in \mathbb{R}$ be such that $0 < \epsilon < 1$. For all $n$ such that $1/ \log(n) < \epsilon$ we have 
\[
m/n \leq (n/\log(n))/n < \epsilon
\]
and thus
\begin{align*}
&n^2 \left(\sum_{\substack{l=n-k \\ l \textrm{ even}}}^{n} \frac{\kappa_O(S_{n-l})}{l^2} + \sum_{\substack{l=n-k \\ l \textrm{ odd}}}^{n} \frac{\kappa_E(S_{n-l})}{l^2}\right)\\
&\hspace{4.5cm}= n^2 \left(\sum_{\substack{m=0 \\ m \textrm{ even}}}^{k} \frac{\kappa_O(S_{m})}{(n-m)^2} + \sum_{\substack{m=1 \\ m \textrm{ odd}}}^{k} \frac{\kappa_E(S_{m})}{(n-m)^2} \right) \\
&\hspace{4.5cm} = \sum_{\substack{m=0 \\ m \textrm{ even}}}^{k} \frac{\kappa_O(S_{m})}{(1-m/n)^2} + \sum_{\substack{m=1 \\ m \textrm{ odd}}}^{k} \frac{\kappa_E(S_{m})}{(1-m/n)^2}\\
&\hspace{4.5cm}\leq \frac{1}{(1-\epsilon)^2}\left(\sum_{\substack{m=0 \\ m \textrm{ even}}}^{k} \kappa_O(S_{m}) + \sum_{\substack{m=1 \\ m \textrm{ odd}}}^{k} \kappa_E(S_{m})\right).
\end{align*}
These remarks give us
\begin{align*}
\limsup_{\substack{n \rightarrow \infty \\ n \textrm{ even}}} n^2 \kappa_E(S_n) 
&\leq \limsup_{\substack{n \rightarrow \infty \\ n \textrm{ even}}} n^2 \kappa_E(S_n)\\ 
&\leq \frac{1}{(1-\epsilon)^2}\left(\sum_{\substack{m=0 \\ l \textrm{ even}}}^{\left\lfloor n/ \log(n)\right\rfloor} \kappa_O(S_{m}) + \sum_{\substack{m=1 \\ l \textrm{ odd}}}^{\left\lfloor n/ \log(n)\right\rfloor} \kappa_E(S_{m})\right)\\
&\leq  \frac{1}{(1-\epsilon)^2}\left(\sum_{\substack{m=0 \\ l \textrm{ even}}}^{\infty} \kappa_O(S_{m}) + \sum_{\substack{m=1 \\ l \textrm{ odd}}}^{\infty} \kappa_E(S_{m})\right)\\
&=\frac{A_{1}}{(1-\epsilon)^2}.
\end{align*}
But as $\epsilon$ was arbitary, we conclude that
\[
\limsup_{\substack{n \rightarrow \infty \\ n \textrm{ even}}} n^2 \kappa_E(S_n)  \leq A_{1}.
\]

Similarly, we find that 
\[
\limsup_{\substack{n \rightarrow \infty \\ n \textrm{ odd}}} n^2 \kappa_E(S_n)  \leq A_{2},
\]
which completes the proof. 
\end{proof}

From Propositions \ref{prop3} and \ref{prop4} we can now conclude that
\[
\lim_{\substack{n \rightarrow \infty \\ n \textrm{ even}}}n^2 \kappa_E(S_n) = \displaystyle A_1
\]
and
\[
\lim_{\substack{n \rightarrow \infty \\ n \textrm{ odd}}}n^2 \kappa_E(S_n) =\displaystyle A_2.
\] 

This concludes the first half of this section. The second half of this section will focus on $Q(S_n)$ using a similar approach to that of the first half. 
%Let's set

%\begin{align*}
%B_{1} & = \sum_{\substack{d=1\\ d \textrm{ odd}}}^{\infty} Q(S_d)\\
%B_{2} & = \sum_{\substack{d=0\\ d \textrm{ even}}}^{\infty} Q(S_d)
%\end{align*}

\begin{prop}\label{prop5}
\[
\liminf_{\substack{n \rightarrow \infty \\ n \textrm{ even}}} n^2 Q(S_n) \geq \displaystyle B_{1},
\]

\[
\liminf_{\substack{n \rightarrow \infty \\ n \textrm{ odd}}} n^2 Q(S_n)\geq \displaystyle B_{2}.
\]
\end{prop}

\begin{proof}
The second half of Proposition \ref{prop2} tells us that  if $k$ is such that $\frac{n}{2}<k\leq n$, then 

$$
Q(S_n) \geq \sum_{\substack{l=k\\l \textrm{ odd}}}^{n} \frac{Q(S_{n-l})}{l^2}=
\begin{cases*}
\displaystyle \sum_{\substack{m=1\\ m \textrm{ odd}}}^{n-k} \frac{Q(S_m)}{(n-m)^{2}}, & \textrm{ for n even,}\\
\displaystyle \sum_{\substack{m=0\\ m \textrm{ even}}}^{n-k} \frac{Q(S_m)}{(n-m)^{2}}, & \textrm{ for n odd.}
\end{cases*}
$$
Hence the result follows in the same way as in Proposition \ref{prop3}. 
\end{proof}

\begin{prop}\label{prop6}
\[
\limsup_{\substack{n \rightarrow \infty \\ n \textrm{ even}}} n^2 Q(S_n) \leq \displaystyle B_{1},
\]

\[
\limsup_{\substack{n \rightarrow \infty \\ n \textrm{ odd}}} n^2 Q(S_n) \leq
\displaystyle B_{2}. 
\]
\end{prop}

\begin{proof}
As before, we let $k=\left\lfloor \frac{n}{\log(n)} \right\rfloor$. By Proposition \ref{prop2}

\[
Q(S_n)\leq s_k(n)^2 + \sum_{\substack{l=k \\ l \textrm{ odd}}}^{n} \frac{Q(S_{n-l})}{l^2}.
\]
We already know that $n s_{k} \to 0$ as $n\to \infty$ and we estimate the remaining sum as before:
\begin{align*}
\sum_{\substack{l=k \\ l \textrm{ odd}}}^{n} \frac{Q(S_{n-l})}{l^2}
& = \sum_{\substack{l=n-k \\ l \textrm{ odd}}}^{n} \frac{Q(S_{n-l})}{l^2}+\sum_{\substack{l=k \\ l \textrm{ odd}}}^{n-k-1} \frac{Q(S_{n-l})}{l^2}\\
& \leq \sum_{\substack{l=n-k \\ l \textrm{ odd}}}^{n} \frac{Q(S_{n-l})}{l^2} +\sum_{l=k}^{n-k-1} \frac{Q(S_{n-l})}{l^2}\\
& \leq \sum_{\substack{l=n-k \\ l \textrm{ odd}}}^{n} \frac{Q(S_{n-l})}{l^2} +\sum_{l=k}^{n-k-1} \frac{C_2}{l^{2}(n-l)^{2}}\\
& \leq \sum_{\substack{l=n-k \\ l \textrm{ odd}}}^{n} \frac{Q(S_{n-l})}{l^2} +\sum_{l=\lceil n/2\rceil}^{n-k-1} \frac{2C_2}{l^{2}(n-l)^{2}} +\frac{C_2}{k^2(n-k)^2}.
\end{align*}
As in Proposition \ref{prop4}, we find that everything in this sum is $o(n^{-2})$ except the term 
$$
\sum_{\substack{l=n-k \\ l \textrm{ odd}}}^{n} \frac{Q(S_{n-l})}{l^2}=
\begin{cases*}
\displaystyle \sum_{\substack{m=1\\ m \textrm{ odd}}}^{k} \frac{Q(S_m)}{(n-m)^{2}}, & \textrm{ for n is even,}\\
\displaystyle \sum_{\substack{m=0\\ m \textrm{ even}}}^{k} \frac{Q(S_m)}{(n-m)^{2}}, & \textrm{ for n is odd.}
\end{cases*}
$$
Hence, similarly to before, we find that for an arbitrary $0<\epsilon<1$

$$
\limsup_{\substack{n\to\infty \\ n \textrm{ even}}}n^2 Q(S_n) \leq \displaystyle \frac{B_{1}}{(1-\epsilon)^{2}},
$$

$$
\limsup_{\substack{n\to\infty \\ n \textrm{ odd}}}n^2 Q(S_n) \leq \displaystyle \frac{B_{2}}{(1-\epsilon)^{2}}
$$
and the proposition follows since $\epsilon$ is arbitrary. 
\end{proof}

From Proposition \ref{prop5} and \ref{prop6} we can now conclude that

\[
\lim_{\substack{n \to\infty\\ n \textrm{ even}}}n^2 Q(S_n) = B_1,
\]

\[
\lim_{\substack{n \to\infty\\ n \textrm{ odd}}}n^2 Q(S_n) = B_2.
\]

\section{Asymptotics of $\kappa(A_n)$}\label{section5}

\begin{lem} We have
$$\kappa(A_n)=\kappa_E(S_n)-2Q(S_n).$$
\end{lem}

\begin{proof}
Let $\sigma, \tau \in S_n$ be chosen uniformly and independently at random. Let

\begin{itemize}
\item AN be the event that both $\sigma$ and $\tau$ lie in $A_n$,
\item CT be the event that they have the same cycle type,
\item SP be the event that they lie in a conjugacy class in $A_n$ which is split, i.e. that they are formed only of parts of odd, distinct length.
\end{itemize}

We note that $\kappa_E(S_n) = \mathbf{P}(CT | AN)$ and that $Q(S_n) = \mathbf{P}(CT \cap SP)$. Moreover,
\[
\kappa(A_n) = \mathbf{P}(CT|AN) \times \left(1-\frac{1}{2}\mathbf{P}(SP|AN \cap CT)\right)
\]

for if the conjugacy class is not split, then they are automatically conjugate in $A_n$ when they have the same cycle type, and if the conjugacy class is split, then the chance they are conjugate is $\frac{1}{2}$ when they have the same cycle type. We have
\[
\mathbf{P}(SP| AN \cap CT) = \frac{\mathbf{P}(SP \cap CT| AN )}{\mathbf{P}(CT|AN)} =  \frac{\mathbf{P}(SP \cap CT| AN )}{\kappa_E(S_n)}.
\]
But SP implies AN, hence
\[
\mathbf{P}(SP \cap CT| AN ) = \frac{\mathbf{P}(SP \cap CT \cap AN)}{\mathbf{P}(AN)} =\frac{\mathbf{P}(SP \cap CT)}{\frac{1}{4}} = 4Q(S_n).  
\]
Thus
\[
\mathbf{P}(SP| AN \cap CT) = \frac{4Q(S_n)}{\kappa_E(S_n)}  
\]
and 
\[
\kappa(A_n)=\kappa_E(S_n) - 2Q(S_n).
\] 
\end{proof}

\begin{proof}[Proof of Theorem \ref{thm1}]

From the above lemma, we have 
\begin{align*}
\lim_{n \to\infty} n^2 \kappa(A_n)
& =\lim_{n\to\infty}\left\{n^2\kappa_E(S_n)-2n^2Q(S_n)\right\}\\
& = \lim_{n\to\infty}\left\{n^2\kappa_E(S_n)\right\} - 2 \lim_{n\to\infty}\left\{n^2Q(S_n)\right\}
\end{align*}
and we can finally conclude that
\[
\lim_{\substack{n \to\infty\\ n \textrm{ even}}} n^2 \kappa(A_n) = A_1-2B_1, 
\]

\[
\lim_{\substack{n \to\infty\\ n \textrm{ odd}}} n^2 \kappa(A_n) = A_2-2B_2.
\]

\end{proof}

\section{Further Remarks}

It is clear that there is still more which can be said about the conjugacy probability of $A_n$, and of other groups. Firstly, it is possible that the methods used by Flajolet \emph{et al} (2006) would allow the explicit calculation the asymptotic values of $\kappa (A_n)$ to greater precision. 

It may also be interesting to consider the conjugacy probability to other families of groups. In particular, the family $\mathrm{GL}(n,q),$ with $q$ fixed and $n$ tending to infinity may be an interesting case to consider.

Lastly, a different but related probability that may be of interest is the probability that one element chosen uniformly at random from $S_n$ belongs to a split conjugacy class in $A_n$ - we called this probability $q(S_n)$. Numerical evidence suggests $n^{-\frac{1}{2}}q(S_n)$ has a limit as $n$ tends to infinity but our brief investigations suggest that this would be harder to prove than the limits treated in this paper.

\section{Acknowledgements}
This work was completed under the supervision of Dr John R. Britnell as part of a UROP placement with the Department of Mathematics at Imperial College London. We would like to thank Dr. Britnell for his constructive suggestions and excellent guidance while we were working on this paper. We would also like to thank the reviewer for their helpful comments.



\end{document}